\documentclass[preprint]{elsarticle}
\usepackage[T1]{fontenc}
\usepackage{lmodern}

\usepackage{amsmath, amsthm, amssymb,amsfonts}					
\usepackage{hyperref}
\usepackage{kbordermatrix}									
\usepackage[usenames,dvipsnames,svgnames,table]{xcolor}
\usepackage{enumitem}
\usepackage{multirow}
\usepackage{float}
\usepackage{subcaption}

\usepackage{tikz}
\usepackage{pgfplots}
\pgfplotsset{compat=newest}

\usepgfplotslibrary{external} 
\everymath{\displaystyle}

\newtheorem{thm}{Theorem}[section]
\newtheorem{lem}[thm]{Lemma}
\newtheorem{cor}[thm]{Corollary}
\newtheorem{prop}[thm]{Proposition}
\newtheorem{conj}[thm]{Conjecture}
\newtheorem{prob}[thm]{Problem}

\theoremstyle{definition}			                						

\newtheorem{rem}[thm]{Remark}
\newtheorem{ex}[thm]{Example}

\numberwithin{equation}{section}		            					

\newcommand{\bb}[1]{\mathbb{#1}}								
\newcommand{\ii}{\textup{i}}									
\newcommand{\mat}[2]{{M}_{#1}(#2)}							

\newcommand{\inv}[1]{#1^{-1}}								
\newcommand\floor[1]{\lfloor#1\rfloor}								

\newcommand{\jord}[2]{J_{#1} \left( #2\right)}

\newcommand{\sig}[1]{\sigma \left( #1 \right)}						
\newcommand{\dg}[1]{\Gamma \left( #1 \right)}						

\newcommand{\hyp}[2]{#1 \hyperref[#2]{\ref*{#2}}}					

\newcommand{\karp}{Karpelevi{\v{c}}}
\journal{Linear Algebra and its Applications}

\begin{document}
\begin{frontmatter}
\title{A matricial view of the \karp~Theorem}

\author[addy1]{Charles R.~Johnson}										
\ead{crjohn@wm.edu}

\author[addy2]{Pietro Paparella\corref{corpp}}
\ead{pietrop@uw.edu}
\ead[url]{http://faculty.washington.edu/pietrop/}

\cortext[corpp]{Corresponding author.}

\address[addy1]{Department of Mathematics, College of William \& Mary, Williamsburg, VA 23187-8795, USA}
\address[addy2]{Division of Engineering and Mathematics, University of Washington Bothell, Bothell, WA 98011-8246, USA}

\begin{abstract}
The question of the exact region in the complex plane of the possible single eigenvalues of all $n$-by-$n$ stochastic matrices was raised by Kolmogorov in 1937 and settled by \karp~in 1951 after a partial result by Dmitriev and Dynkin in 1946. The \karp~result is unwieldy, but a simplification was given by {\DJ}okovi{\'c} in 1990 and Ito in 1997. The \karp~region is determined by a set of boundary arcs each connecting consecutive roots of unity of order less than $n$. It is shown here that each of these arcs is realized by a single, somewhat simple, parameterized stochastic matrix. Other observations are made about the nature of the arcs and several further questions are raised. The doubly stochastic analog of the \karp~region remains open, but a conjecture about it is amplified.
\end{abstract}

\end{frontmatter}

\section{Introduction}

In \cite{k1937}, Kolmogorov posed the problem of characterizing the subset of the complex plane, denoted by $\Theta_n$, that consists of the individual eigenvalues of all $n$-by-$n$ stochastic matrices. 

One can easily verify that for each $n \geq 2$, the region $\Theta_n$ is closed, inscribed in the unit-disc, star-convex (with star-centers at zero and one), and symmetric with respect to the real-axis. Furthermore, it is clear that $\Theta_n \subseteq \Theta_{n+1}$, $\forall n \in \bb{N}$. In view of these properties, $\partial \Theta_n = \{ \lambda \in \Theta_n : \alpha \lambda \not \in \Theta_n,\forall \alpha > 1\}$, and each region is determined by its boundary.   

Dmitriev and Dynkin \cite{dd1946} obtained a partial solution to Kolmogorov's problem, and \karp~\cite[Theorem B]{k1951}, expanding on the work of \cite{dd1946}, resolved it by showing that the boundary of $\Theta_n$ consists of curvilinear arcs (herein, \emph{\karp~arcs} or K-arcs), whose points satisfy a polynomial equation that is determined by the endpoints of the arc (which are consecutive roots of unity). {\DJ}okovi{\'c} \cite[Theorem 4.5]{d1990} and Ito \cite[Theorem 2]{i1997} each provide a simplification of this result. However, noticably absent in the \karp~Theorem (and the above-mentioned works) are \emph{realizing-matrices} (i.e., a matrix whose spectrum contains a given point) for points on these arcs. 

This problem has been addressed previously in the literature. Dmitriev and Dynkin \cite[Basic Theorem]{dd1946} give a schematic description of such matrices for points on the boundary of $\Theta_n \backslash \Theta_{n-1}$ and Swift \cite[\S 2.2.2]{s1972} provides such matrices for $3\leq n \leq 5$.

Our main result is providing, for every $n$ and for each arc, a single parametric matrix that realizes the entire $K$-arc as the parameter runs from 0 to 1. Aside from the theoretical importance -- after all, the original problem posed by Kolmogorov is intrinsically matricial -- possession of such matrices is instrumental in the study of nonreal \emph{Perron similarities} in the longstanding \emph{nonnegative inverse eigenvalue problem} \cite{jp2017}, and provides a framework for resolving Conjecture 1 \cite{lpk2015} vis-\`{a}-vis the results in \cite{j1981}. 

In addition, we provide some partial results on the differentiability of the Karpelevi{\v{c} arcs. We demonstrate that some powers of certain realizing-matrices realize other arcs. Finally, we pose several problems that appeal to a wide variety of mathematical interests.

\section{Notation \& Background}

The algebra of complex (real) $n$-by-$n$ matrices is denoted by $\mat{n}{\bb{C}}$ ($\mat{n}{\bb{R}}$). A real matrix is called \emph{nonnegative} (\emph{positive}) if it is an entrywise nonnegative (positive) matrix. If $A$ is nonnegative (positive), then we write $A \geq 0$ ($A > 0$). 

An $n$-by-$n$ nonnegative matrix $A$ is called \emph{(row) stochastic} if every row sums to unity; \emph{column stochastic} if every column sums to unity; and \emph{doubly stochastic} if it is row stochastic and column stochastic.

Given $n \in \bb{N}$, the set ${F}_n := \{ p/q : 0\leq p < q \leq n,~\gcd(p,q)=1 \}$ is called the \emph{set of Farey fractions of order n}. If $p/q$, $r/s$ are elements of ${F}_n$ such that $p/q < r/s$, then $(p/q,r/s)$ is called a \emph{Farey pair (of order $n$)} if $x \not\in {F}_n$ whenever $p/q < x < r/s$. The Farey fractions $p/q$ and $r/s$ are called \emph{Farey neighbors} if $(p/q,r/s)$ or $(r/s, p/q)$ is a Farey pair.

The following is the celebrated \karp~Theorem in a form due to Ito \cite{i1997}. 

\begin{thm}[\karp] 
\label{thm:karpito}
The region $\Theta_n$ is symmetric with respect to the real axis, is included in the unit-disc $\{ z \in \bb{C} : |z| \leq 1\}$, and intersects the unit-circle $\{ z \in \bb{C} : |z| = 1\}$ at the points $\{ e^{2\pi\ii p/q} : p/q \in {F}_n \}$. The boundary of $\Theta_n$ consists of these points and of curvilinear arcs connecting them in circular order. 

Let the endpoints of an arc be $e^{2\pi\ii p/q}$ and $e^{2\pi\ii r/s}$ ($q < s$). Each of these arcs is given by the following parametric equation:  
\begin{equation}
t^{s} \left( t^{q} - \beta \right)^{\floor{n/q}} = \alpha^{\floor{n/q}} t^{q\floor{n/q}},~\alpha \in [0,1], ~\beta:=1-\alpha \label{ito_eq}.
\end{equation} 
\end{thm}

\hyp{Figure}{fig:karpregions} contains the regions $\Theta_3$, $\Theta_4$, and $\Theta_5$.

\begin{figure}[H]
\centering
\begin{subfigure}{.32\textwidth}\centering
\begin{tikzpicture}
\begin{axis}[
axis lines=none,
axis equal image,
scale=0.32,
xlabel={$\Re{(\lambda)}$},
ylabel={$\Im{(\lambda)}$},
ylabel style={rotate=-90,, anchor=north},
xmin=-1,
xmax=1,
ymin=-1.0,
ymax=1.0,
xtick={-1,1},
ytick={-1,1}
]

\addplot[thick,black] coordinates{
(1,0) 
(-.5,.866025403784439) 
(-.5,-.866025403784439)
(1,0)}; 
\addplot[thick,black] coordinates{(-1,0) (-.5,0)};                              				

\draw[color=gray] (axis cs:0,0) circle (1);
\end{axis}
\end{tikzpicture}
\caption{$\Theta_3$}
\label{fig:karpregion3}
\end{subfigure}
\hfill
\begin{subfigure}{.32\textwidth}\centering
\begin{tikzpicture}
\begin{axis}[
axis lines=none,
axis equal image,
scale=0.32,
xlabel={$\Re{(\lambda)}$},
ylabel={$\Im{(\lambda)}$},
ylabel style={rotate=-90, anchor=north},
xmin=-1,
xmax=1,
ymin=-1.0,
ymax=1.0,
xtick={-1,1},
ytick={-1,1}
]

\addplot[thick,black] coordinates{(1,0) (0,1)};                               
\addplot[thick,black] coordinates{(1,0) (0,-1)};                              
\addplot[thick,black] table {aarc1a.dat};                                     
\addplot[thick,black] table {aarc1b.dat};                                     
\addplot[thick,black] table {aarc2a.dat};                                     
\addplot[thick,black] table {aarc2b.dat};                                     

\draw[color=gray] (axis cs:0,0) circle (1);
\end{axis}
\end{tikzpicture}
\caption{$\Theta_4$}
\label{fig:karpregion4}
\end{subfigure}
\hfill
\begin{subfigure}{.32\textwidth}\centering
\begin{tikzpicture}
\begin{axis}[
axis lines=none,
axis equal image,
scale=0.32,
xlabel={$\Re{(\lambda)}$},
ylabel={$\Im{(\lambda)}$},
ylabel style={rotate=-90, anchor=north},
xmin=-1,
xmax=1,
ymin=-1.0,
ymax=1.0,
xtick={-1,1},
ytick={-1,1}
]

\addplot[thick,black] coordinates{(1,0) (.309016994374947,.951056516295154)};                               
\addplot[thick,black] coordinates{(1,0) (.309016994374947,-.951056516295154)};                               
\addplot[thick,black] table {fivearcs1.dat};                                     						
\addplot[thick,black] table {fivearcs2.dat};                                     
\addplot[thick,black] table {fivearcs3.dat};                                     
\addplot[thick,black] table {fivearcs4.dat};                                     
\addplot[thick,black] table {fivearcs5.dat};                                     
\addplot[thick,black] table {fivearcs6.dat};                                     
\addplot[thick,black] table {aarc1a.dat};                                     
\addplot[thick,black] table {aarc1b.dat};                                     

\draw[color=gray] (axis cs:0,0) circle (1);
\end{axis}
\end{tikzpicture}
\caption{$\Theta_5$}
\label{fig:karpregion5}
\end{subfigure}
\caption{$\Theta_n$, $3 \leq n \leq 5$}
\label{fig:karpregions}
\end{figure}
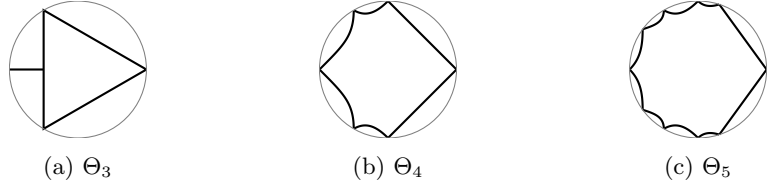

For $n \in \bb{N}$, we call the collection of such arcs \emph{the K-arcs (of order $n$)} and we denote by $K(p/q,r/s) = K_n(p/q,r/s)$ the arc connecting $e^{2 \pi \ii p /q}$ and $e^{2 \pi \ii r /s}$, when $p/q$ and $r/s$ are Farey neighbors. Notice that the number of K-arcs equals $| F_n| = 1 + \sum_{k=1}^n \phi(k)$, where $\phi$ denotes \emph{Euler's totient function}. 

For Farey neighbors $p/q$ and $r/s$, $q < s$, we call the collection of equations \eqref{ito_eq} the \emph{Ito equations (with respect to $\{p/q,r/s\}$)} and the collection of polynomials 
\[ f_\alpha (t) := t^{s} \left( t^{q} - \beta \right)^{\floor{n/q}} - \alpha^{\floor{n/q}} t^{q\floor{n/q}},~\alpha \in [0,1]\] 
the \emph{Ito polynomials (with respect to $\{p/q,r/s\}$)}.   

A \emph{directed graph} (or simply \emph{digraph}) $\Gamma = (V,E)$ consists of a finite, nonempty set $V$ of \emph{vertices}, together with a set $E \subseteq V \times V$ of \emph{arcs}. For $A \in \mat{n}{\bb{C}}$, the \emph{directed graph} (or simply \emph{digraph}) of $A$, denoted by $\Gamma = \dg{A}$, has vertex set $V = \{ 1, \dots, n \}$ and arc set $E = \{ (i, j) \in V \times V : a_{ij} \neq 0\}$. 

A digraph $\Gamma$ is called \emph{strongly connected} if for any two distinct vertices $i$ and $j$ of $\Gamma$, there is a path in $\Gamma$ from $i$ to $j$. Following \cite{br1991}, we consider every vertex of $V$ as strongly connected to itself. A strong digraph is \emph{primitive} if the greatest common divisor of all its cycle-lengths is one, otherwise it is \emph{imprimitive}. 

For $n \geq 2$, an $n$-by-$n$ matrix $A$  is called \emph{reducible} if there exists a permutation matrix $P$ such that
\begin{align*}
P^\top A P =
\begin{bmatrix}
A_{11} & A_{12} \\
0 & A_{22}
\end{bmatrix},
\end{align*}
where $A_{11}$ and $A_{22}$ are nonempty square matrices. If $A$ is not reducible, then A is called \emph{irreducible}. It is well-known that a matrix $A$ is irreducible if and only if $\dg{A}$ is strongly connected (see, e.g., \cite[Theorem 3.2.1]{br1991} or \cite[Theorem 6.2.24]{hj2013}). 

An irreducible nonnegative matrix is called \emph{primitive} if, in its digraph, the set of cycle-lengths is relatively prime; otherwise it is \emph{imprimitive}. 

For $n \in \bb{N}$, denote by $C_n$ the \emph{basic circulant}, i.e.,  
\[ C_n =
\left[ 
\begin{array}{cc}
0 & I_{n-1} \\
1 & 0
\end{array} \right]. \]
Note that the digraph of $C_n$ is a cycle of length $n$.

Given an $n$-by-$n$ matrix $A$, the \emph{characteristic polynomial of $A$}, denoted by $\chi_A$, is defined by $\chi_A = \det{(tI - A)}$. The \emph{companion matrix} $C = C_f$ of a monic polynomial $f(t) = t^n + \sum_{k=1}^{n} c_{k} t^{n - k}$ is the $n$-by-$n$ matrix defined by
\[ C = 
\left[\begin{array}{cc}
0 & I_{n-1} \\
-c_n & -c
\end{array} \right], \]
where $c = [c_{n-1}~\cdots~c_1]$. It is well-known that $\chi_C = f$. Notice that $C$ is irreducible if and only if $c_n \neq 0$. 

%

\section{Realizing-matrices}

\begin{lem}\label{lem:det}
Let $A \in \mat{n}{\bb{C}}$. If $B = A + \alpha e_k e_\ell^\top$, then $\det(B) = \det(A) + (-1)^{k + \ell} \alpha \det(A_{k\ell})$.
\end{lem}

\begin{proof}
Take either a Laplace-expansion along the $k$-th row or the $k$-th column of $B$. 
\end{proof} 

\begin{thm}\label{thm:main}
For each K-arc $K_n (p/q,r/s)$, there is a parametric, stochastic matrix $M = M(\alpha)$, $0\leq \alpha \leq 1$, such that each point $\lambda = \lambda(\alpha)$ of the arc is an eigenvalue of $M$. Furthermore, if $\alpha \in (0,1)$, then $M$ is primitive.
\end{thm}

\begin{proof}
Let $p/q$ and $r/s$ be Farey neighbors, where $q < s$. Note that $s \neq q\floor{n/q}$ since $q$ and $s$ are relatively prime. 

First, we consider the case in which  $p/q = 0$ and $r/s = 1/n$ (which we call the \emph{Type 0 arc}). Then \eqref{ito_eq} reduces to $(t  - \beta)^n - \alpha^n = 0$. If 
\[ M = M(\alpha) := \alpha  C_n + \beta I  
= 
\begin{bmatrix}
\beta & \alpha & 	 	\\
 & \beta & \alpha 		\\
 & & \ddots & \ddots 	\\
 & & & \beta & \alpha 	\\
\alpha & & & & \beta
\end{bmatrix} \in \mat{\bb{R}}{n}, \]
then 
\begin{align*} 
\chi_M (t) 
&= \det{(tI - (\alpha  C_n + \beta I))} \\
&= \det{((t - \beta)I - \alpha  C_n)} \\
& = \chi_{\alpha C_n} (t - \beta) \\
& = (t - \beta)^n - \alpha^n.
\end{align*} 
If $\alpha \in (0,1)$, then $\dg{M}$ contains directed-cycles of length one and $n$. Hence, $M$ is irreducible and since the greatest common divisor of all cycle-lengths of $\dg{M}$ is obviously one, $M$ is primitive.

Next, we consider the case in which $\floor{n/q}= 1$ (herein referred to as a \emph{Type I arc}). Then \eqref{ito_eq} reduces to $t^{s}  - \beta  t^{s-q} - \alpha = 0$. If 
\begin{align} 
M = M(\alpha) :=  
\begin{bmatrix}
z & I \\
\alpha & \beta e_{s-q}^\top 
\end{bmatrix} \in \mat{s}{\bb{R}}, \label{typeonemats}
\end{align} 
then $M \geq 0$ and $\chi_{M} (t) = t^{s}  - \beta  t^{s-q} - \alpha$. If $\alpha \in (0,1)$, then $\dg{M}$ contains $\dg{C_n}$. Hence, $M$ is irreducible, and, since $\gcd{(s-(s-q),s)} = \gcd{(q,s)}=1$, it must be primitive.

Next, we consider the case in which $\floor{n/q}> 1$ and $s < q\floor{n/q}$ (which we call a \emph{Type II arc}). Then \eqref{ito_eq} reduces to 
\begin{align*}
(t^q  - \beta)^{\floor{n/q}} - \alpha^{\floor{n/q}} t^{q\floor{n/q} - s} = 0.
\end{align*}
Consider the nonnegative matrix $M = M(\alpha) := \alpha X +  \beta Y$, where $X$ is the nonnegative companion matrix of the polynomial $t^{q\floor{n/q}} - t^{q\floor{n/q} - s}$, and
\begin{align*} 
Y := \bigoplus_{k=1}^{\floor{n/q}} C_q =
\begin{bmatrix}
C_q & \\
 & \ddots & \\
 & & C_q
\end{bmatrix} \in \mat{q\floor{n/q}}{\bb{R}}.
\end{align*} 
Since $1 < q\floor{n/q} - s + 1 \leq n - s + 1 < q + 1$, it follows that 
\begin{align*} 
M =
\left[ 
\begin{array}{*{14}{c}}
& 1 & & & \vline & & & & & \vline 											\\
& & \ddots & & \vline & & & & & \vline 										\\
& & & 1 & \vline & & & & & \vline 											\\
\beta & & & & \vline & \alpha & & & & \vline 									\\
\hline 
& & & & \vline & \multicolumn{4}{c}{\multirow{4}{*}{\Large $\ddots$}} & \vline & & &  	\\
& & & & \vline & & & & & \vline & & &  										\\
& & & & \vline & & & & &\vline & & &  										\\
& & & & \vline & & & & & \vline & \alpha 										\\
\hline
& & & & \vline & & & & & \vline & & 1 										\\
& & & & \vline & & & & & \vline & & & \ddots 									\\
& & & & \vline & & & & & \vline & & & & 1 									\\
\multicolumn{4}{c}{\alpha e_{q\floor{n/q} - s + 1}^\top} & \vline & & & & & \vline & \beta & &   
\end{array} 
\right],
\end{align*}
where $e_{q\floor{n/q} - s + 1} \in \bb{R}^q$. Because $M - \alpha e_{q\floor{n/q}} e_{q\floor{n/q} - s + 1}^\top$ is block upper-triangular, it follows from \hyp{Lemma}{lem:det} that 
\begin{align*}
\chi_M (t) 
&= (t^q - \beta)^{\floor{n/q}} + 																			\\
&\quad (-1)^{2 q\floor{n/q} -s + 1} (-\alpha) t^{q\floor{n/q} - s} (-\alpha)^{\floor{n/q} -1} (-1)^{q\floor{n/q} - 1 - (q\floor{n/q} - s) - (\floor{n/q} -1)}	\\
&= (t^q - \beta)^{\floor{n/q}} + (-1)^{2 q\floor{n/q}+ 1} \alpha t^{q\floor{n/q} - s} 											\\
&= (t^q - \beta)^{\floor{n/q}} - \alpha t^{q\floor{n/q} - s}.
\end{align*}
If $\alpha \in (0,1)$, then the directed graph contains $\floor{n/q}$ strongly connected components and the graph on these components, determined whether off-diagonal blocks are nonzero, is also strongly connected; hence, the entire graph is strongly connected, i.e., $M$ is irreducible. Furthermoe, since $\dg{M}$ contains cycles of length $q$ and $q\floor{n/q} - (q\floor{n/q} - s + 1) + 1 = s$, it follows that $M$ is primitive .

Finally, we consider the case when $\floor{n/q}> 1$ and $s > q\floor{n/q}$ (herein referred to as a \emph{Type III arc}). For convenience, let $d = s - q\floor{n/q}$. Then \eqref{ito_eq} reduces to 
\begin{align*}
t^{d} (t^q  - \beta)^{\floor{n/q}} - \alpha^{\floor{n/q}}  = 0.
\end{align*}
Consider the nonnegative matrix $M = M(\alpha) := \alpha C_s + \beta Y$, where
\begin{align*} 
Y =
\begin{bmatrix}
\jord{d}{0}  & 				\\
		& C_q & & 			\\
		&        & \ddots & 		\\
		& 	&	    & C_q
\end{bmatrix} + e_d e_{d+1}^\top \in \mat{s}{\bb{R}}.
\end{align*}
Then 
\[ 
M =
\kbordermatrix{& & & & & d & & 													\\
  & 0 & 1 & & & & \vline & & & & & \vline & & & & & \vline & & & & 								\\
  & & 0 & 1 & & & \vline & & & & & \vline & & & & & \vline & & & &								\\
  & & & \ddots & \ddots & & \vline & & & & & \vline & & & & & \vline & & & &						\\
  & & & & 0 & 1 & \vline & & & & & \vline & & & & & \vline & & & & 								\\
d & & & & & 0 & \vline & 1 & & & & \vline & & & & & \vline& & & & 								\\
\cline{2-20}
& & & & & & \vline & & 1 & & & \vline & & & & & \vline & & & & 								\\
& & & & & & \vline & & & \ddots  & & \vline & & & & & \vline & & & &							\\
& & & & & & \vline & & & & 1 & \vline & & & & & \vline & & & &								\\
& & & & & & \vline & \beta & & & & \vline & \alpha & & & & \vline & & & & 							\\
\cline{2-20}
& & & & & & \vline & & & & & \vline & \multicolumn{4}{c}{\multirow{4}{*}{\Large $\ddots$}} & \vline & & & 	\\
& & & & & & \vline & & & & & \vline & & & & & \vline										\\
& & & & & & \vline & & & & & \vline & & & & & \vline										\\
& & & & & & \vline & & & & & \vline & & & & & \vline & \alpha									\\
\cline{2-20}
& & & & & & \vline & & & & & \vline & & & & & \vline & & 1 & & 								\\
& & & & & & \vline & & & & & \vline & & & & & \vline & & & \ddots &  							\\
& & & & & & \vline & & & & & \vline & & & & & \vline & & &  & 1 								\\
& \alpha & & & & & \vline & & & & & \vline & & & & & \vline & \beta & & & }.
\]
Since $M - \alpha e_s e_1^\top$ is block upper-triangular, following \hyp{Lemma}{lem:det}, 
\begin{align*}
\chi_M (t) 
&= t^d (t^q - \beta)^{\floor{n/q}} + (-1)^{s+1}(-\alpha)(-\alpha)^{\floor{n/q}-1}(-1)^{s - 1 -(\floor{n/q}-1)} 	\\
&= t^d (t^q - \beta)^{\floor{n/q}} + (-1)^{2s+1}\alpha^{\floor{n/q}}							 \\
&= t^d (t^q - \beta)^{\floor{n/q}} - \alpha^{\floor{n/q}}.
\end{align*}
If $\alpha \in (0,1)$, then $\dg{M}$ contains $\dg{C_n}$ as a subgraph. Hence, $M$ is irreducible, and since $\dg{M}$ clearly contains cycles of length $q$ and $s$, $M$ is primitive.
\end{proof}

\begin{rem}
Notice that the realizing matrices for arcs of Type I, II, and II all have trace zero. 
\end{rem}

\begin{ex}
\hyp{Table}{tabone} contains realizing matrices illustrating each type of arc when $n=9$ (the smallest order for which each arc-type appears). 

\begin{table}[H]
\centering
\begin{tabular}{ccc}
$K\left(\frac{p}{q},\frac{r}{s} \right)$ & \emph{Type} & $M(\alpha)$, $\beta := 1 - \alpha$	\vspace*{5pt} \\
$K\left(\frac{1}{9},\frac{1}{8} \right)$ & I & 
$\begin{bmatrix}
0 & 1 & 0 & 0 & 0 & 0 & 0 & 0 & 0 \\
0 & 0 & 1 & 0 & 0 & 0 & 0 & 0 & 0 \\
0 & 0 & 0 & 1 & 0 & 0 & 0 & 0 & 0 \\
0 & 0 & 0 & 0 & 1 & 0 & 0 & 0 & 0 \\
0 & 0 & 0 & 0 & 0 & 1 & 0 & 0 & 0 \\
0 & 0 & 0 & 0 & 0 & 0 & 1 & 0 & 0 \\
0 & 0 & 0 & 0 & 0 & 0 & 0 & 1 & 0 \\
0 & 0 & 0 & 0 & 0 & 0 & 0 & 0 & 1 \\
\alpha & \beta & 0 & 0 & 0 & 0 & 0 & 0 & 0 
\end{bmatrix}$ \vspace*{5pt}								\\ 
$K\left(\frac{2}{7},\frac{1}{3} \right)$ & II & 
$ 
\left[ \begin{array}{*{11}{c}}
0 & 1 & 0 & \vline &  0 & 0 & 0 & \vline & 0 & 0 & 0 \\
0 & 0 & 1 &  \vline &0 & 0 & 0 & \vline & 0 & 0 & 0 \\
\beta & 0 & 0 & \vline & \alpha & 0 & 0 & \vline & 0 & 0 & 0 \\
\cline{1-11}
0 & 0 & 0 & \vline & 0 & 1 & 0 & \vline & 0 & 0 & 0 \\
0 & 0 & 0 &  \vline &0 & 0 & 1 & \vline & 0 & 0 & 0 \\
0 & 0 & 0 &  \vline &\beta & 0 & 0 & \vline & \alpha & 0 & 0 \\
\cline{1-11}
0 & 0 & 0 &  \vline &0 & 0 & 0 & \vline & 0 & 1 & 0 \\
0 & 0 & 0 &  \vline &0 & 0 & 0 & \vline & 0 & 0 & 1 \\
0 & 0 & \alpha & \vline & 0 & 0 & 0 & \vline & \beta & 0 & 0
\end{array} \right]$ \vspace*{5pt} 							\\ 
$K\left(\frac{2}{9},\frac{1}{4} \right)$ & III & 
$
\left[ \begin{array}{*{11}{c}}
0 & \vline & 1 & 0 & 0 & 0 & \vline & 0 & 0 & 0 & 0 \\
\cline{1-11}
0 & \vline & 0 & 1 & 0 & 0 & \vline & 0 & 0 & 0 & 0 \\
0 & \vline & 0 & 0 & 1 & 0 & \vline & 0 & 0 & 0 & 0 \\
0 & \vline & 0 & 0 & 0 & 1 & \vline & 0 & 0 & 0 & 0 \\
0 & \vline & \beta & 0 & 0 & 0 & \vline & \alpha & 0 & 0 & 0 \\
\cline{1-11}
0 & \vline & 0 & 0 & 0 & 0 & \vline & 0 & 1 & 0 & 0 \\
0 & \vline & 0 & 0 & 0 & 0 & \vline & 0 & 0 & 1 & 0 \\
0 & \vline & 0 & 0 & 0 & 0 & \vline & 0 & 0 & 0 & 1 \\
\alpha & \vline & 0 & 0 & 0 & 0 & \vline & \beta & 0 & 0 & 0
\end{array} \right]$
\end{tabular}
\caption{Realizing matrices for arcs of Type I, II, and II when $n=9$.}
\label{tabone}
\end{table}

Let $\mathcal{M} := \{ M(\alpha) : \alpha \in [0,1] \}$ be the set of realizing matrices for the arc $K(1/9,1/8)$. For $d \in \bb{N}$, let $\mathcal{M}^d = \{ M(\alpha)^d : M(\alpha) \in \mathcal{M} \}$. \hyp{Theorem}{thm:arcpowers} shows that certain powers of the realizing matrices for the arc realize other arcs: in particular, $\mathcal{M}^2$, $\mathcal{M}^3$, and $\mathcal{M}^4$ form a set of realizing matrices for the arcs $K(2/9,1/4)$, $K(1/3,3/8)$, and $K(4/9,1/2)$, respectively. 
\end{ex}

\section{Differentiability of the Arcs}
We investigate here the smoothness of the K-arcs, a natural question not previously addressed. 

To that end, let $f$ and $g$ be monic polynomials of degree $n$. For $\alpha \in [0,1]$, let $c_\alpha := \alpha f + (1 - \alpha) g$. Since the roots of a polynomial vary continuously with respect to its coefficients, it follows that the locus $L(f,g) := \left\{ t \in \bb{C}: c_\alpha(t) = 0,~\alpha \in [0,1] \right\}$ consists of $n$ continuous paths (counting multiplicities), each of which connects a root of $g$ to a root of $f$, whose points depend continuously on the parameter $\alpha$ (if $f$ and $g$ share a root, then there is a degenerate path at this root). 


Denote by $P(\mu, \lambda)$ the path that starts at the root $\mu$ of $g$ and terminates at the root $\lambda$ of $f$ ($\mu \neq \lambda$). If $r = r(\alpha) \in P(\mu, \lambda)$, $\alpha \in (0,1)$, then 
\begin{align*}
0 = \alpha f(r) + (1 - \alpha) g(r).
\end{align*}
Differentiating with respect to $\alpha$ yields
\begin{align*}
0 = f(r) + \alpha g'(r) r' - g(r) + (1-\alpha)g'(r)r' = f(r) - g(r) + r' c'_\alpha (r).
\end{align*}
If $c_\alpha'(r) \neq 0$ (i.e., if $r$ is not a multiple root of $c_\alpha$), then 
\begin{equation*}
r' = \frac{g(r) - f(r)}{c_\alpha'(r)}.
\end{equation*} 
Thus, the path $P(\mu, \lambda)$ is differentiable at $r$ if $r$ is not a multiple root of $c_\alpha$ \cite{i2011}.

\begin{prop}\label{distincteigs}
For $n \geq 4$, let 
\begin{equation}
f_\alpha (t) := t^n - \beta t - \alpha, ~\alpha \in [0,1], ~\beta := 1 - \alpha. \label{polyalpha}
\end{equation} 
\begin{enumerate}
[label=(\roman*)]
\item If $n$ is even, then $f_\alpha$ has $n$ distinct roots.
\item If $n$ is odd and $\alpha \geq \beta$,  then $f_\alpha$ has $n$ distinct roots.
\item If $n$ is odd and $\alpha < \beta$,  then $f_\alpha$ has a multiple root if and only if
\[ n^n \alpha^{n-1} - (n-1)^{n-1} \beta^n = n^n \alpha^{n-1} + (n-1)^{n-1} (\alpha -1)^n = 0. \]  
\end{enumerate}
\end{prop}

\begin{proof} Notice that $f_\alpha(1) = 0$, and, since $C_f$ is primitive, if $f_\alpha(\lambda) = 0$, $\lambda \neq 1$, then 
\begin{equation}
|\lambda| < 1. \label{dominate}
\end{equation} 

It is well-known that a polynomial has a multiple root if and only if it shares a root with its formal derivative. Thus, $f_\alpha$ has a multiple root $\lambda \in \bb{C}$ if and only if $f_\alpha (\lambda) = f_\alpha'(\lambda) = 0$, i.e., if and only if  
\begin{align}
\lambda^n - \beta \lambda - \alpha &= 0 	\label{poly}	\\
n\lambda^{n-1} - \beta &= 0			\label{deriv}.
\end{align}
Solving for $\beta$ in \eqref{deriv} and substituting the result in \eqref{poly} yields
\begin{equation}
\lambda^n = -\frac{\alpha}{n-1}	.		\label{lambdan}
\end{equation}
Substituting for $\lambda^n$ in \eqref{poly} yields
\begin{equation}
\lambda = -\frac{\alpha n}{\beta(n-1)} < 0. 	\label{lambdafrac}
\end{equation} 

We now consider each part separately:

\begin{enumerate}
[label=(\roman*)]
\item For contradiction, if $f_\alpha$ has a multiple root, then it must be negative \eqref{lambdafrac}; however, because $f_\alpha(-t) = t^n + \beta t - \alpha$, Descartes' Rule of Signs ensures that $p$ has at most one negative root, a contradiction. 
 
\item Suppose that $n$ is odd and $\alpha \geq \beta$. For contradiction, if $f_\alpha$ has a multiple root, then, following \eqref{lambdafrac},  
\begin{equation*}
|\lambda| = \frac{\alpha n}{\beta(n-1)} \geq \frac{n}{n-1} > 1,
\end{equation*}
contradicting \eqref{dominate}.

\item It is well-known that a polynomial $f$ has a multiple root if and only if its \emph{resultant} $R(f,f')$ vanishes. If 
\[ 
S(f_\alpha,f_\alpha') =
\kbordermatrix{
   & 1 & \cdots & n-1 & & n & n+1 & & 2n-1 		\\
1 & 1                         & & & \vrule & -\beta & -\alpha 	\\
\vdots & & \ddots & & \vrule & & \ddots & \ddots 	\\
n-1 & & & 1 & \vrule & & & -\beta & -\alpha 		\\
\cline{2-9}
n & n & & & \vrule & -\beta 				 \\
\vdots & & \ddots & & \vrule & & \ddots 		\\
2n & & & n & \vrule & & & -\beta \\
2n-1 & 0  & \cdots & 0 & \vrule & n & & & -\beta
}, \] 
then $R(f_\alpha,f_\alpha') = | S(f_\alpha,f_\alpha')| = | D - C B|$, where $B$, $C$, and $D$ denote the upper-right, lower-left, and lower-right blocks of $S(f_\alpha,f_\alpha')$. Since 
\[ D - CB =
\begin{bmatrix}
(n-1) \beta & n\alpha & &  \\
& \ddots & \ddots & \\
& & (n-1) \beta & n\alpha \\
n & & & -\beta
\end{bmatrix}, \]
and $n$ is odd, it follows that $R(f_\alpha,f_\alpha')  = n^n \alpha^{n-1} - (n-1)^{n-1} \beta^n$ and the result is established.
\qedhere
\end{enumerate}
\end{proof}

\begin{rem}\label{rem:negmult}
If $n$ is odd and $f_\alpha$ has a multiple root $\lambda$ (which, folllowing \eqref{lambdafrac}, must be negative), then Descartes' Rule of Signs applied to $f_\alpha (-t) = -t^n +\beta t - \alpha$ forces the multiplicity of $\lambda$ as a root of $f_\alpha$ to be exactly two. 
\end{rem}

\begin{rem}
\label{polypi}
Under the hypotheses of part (iii) of \hyp{Proposition}{distincteigs}, the resultant $R(f_\alpha,f'_\alpha) = \pi(\alpha) = n^n \alpha^{n-1} + (n-1)^{n-1} (\alpha -1)^n$ for the polynomial $f_\alpha$ defined in \eqref{polyalpha}, is a univariate polynomial in $\alpha$. Since $\pi(0) = -(n-1)^{n-1} < 0$ and $\pi(1) = n^n > 0$, it folllows that $\pi$ must have a root in $(0,1)$. However, $\pi'(\alpha) = n^n (n-1) \alpha^{n-2} + (n-1)^{n-1} (\alpha -1)^{n-1}$ and because $n$ is odd, we have $\pi(\alpha) \geq 0$ for all $\alpha \geq 0$. Thus, $\pi$ is strictly increasing on $(0,\infty)$ and hence has exactly one root in $(0,1)$. 
\end{rem}
 
\begin{cor} \label{cor:diffarcs}
Let $n \geq 4$ be a positive integer.
\begin{enumerate}
[label=(\roman*)]
\item If $n$ is even and $\floor{n/2} \leq m \leq n$, then the K-arc $K_n \left( {1}/{m},{1}/{m-1} \right)$ is differentiable. 
\item If $n$ is odd and $\floor{n/2}+1 \leq m \leq n$, then the K-arc $K_n \left( {1}/{m},{1}/{m-1} \right)$ is differentiable. 
\end{enumerate}
\end{cor}

\begin{proof}
In view of \hyp{Proposition}{distincteigs}, it suffices to consider the case when $n$ is odd and $\alpha < \beta$, where $f_\alpha$ is defined as in \eqref{polyalpha}; however, this case is clear as well since \hyp{Remark}{rem:negmult} ensures that if $f_\alpha$ has a multiple multiple root, then $\lambda$ is real.   
\end{proof}

\section{Powers of Realizing-matrices}

For each of the arc types listed in the proof of \hyp{Theorem}{thm:main}, we refer to the collection of polynomials 
\begin{align}
f_\alpha(t) &= (t  - \beta)^n - \alpha^n 								\tag{Type 0}		\\
f_\alpha(t) &= t^{s}  - \beta  t^{s-q} - \alpha 							\tag{Type I}		\\
f_\alpha(t) &= (t^q  - \beta)^{\floor{n/q}} - \alpha^{\floor{n/q}} t^{q\floor{n/q} - s}	\tag{Type II} 	\\
f_\alpha(t) &= t^{s - q\floor{n/q}} (t^q  - \beta)^{\floor{n/q}} - \alpha^{\floor{n/q}} 			\tag{Type III}	
\end{align}
as the \emph{reduced Ito polynomials}. 

The following result is readily deduced  from several well-known theorems concerning Farey pairs (see, e.g., \cite[pp. 28--29]{hw2008}).

\begin{lem} 
\label{lem:farey_pair}
If $p/q$, $r/s$ are elements of ${F}_n$, then $(p/q,r/s)$ is a Farey pair of order $n$  if and only if $qr- ps= 1$ and $q + s > n$. 
\end{lem}

\begin{lem}
\label{lem:divisor}
If $d$ is a positive integer such that $1 < d < n$, then $(d/n,d/n-1)$ is a Farey pair of order $n$ if and only if $d$ divides $n$ or $d$ divides $n-1$. 
\end{lem}

\begin{proof}
If there is a positive integer $k$ such that $n = dk$, then $(d/n,d/n-1) = (1/k,d/n-1)$. Since $dk - (n-1) = 1 $, it follows that $d/n-1 \in {F}_n$.  Because $k > 1$, it follows that $k + n - 1 >  n$. Following \hyp{Lemma}{lem:farey_pair}, $(1/k,d/n-1)$ is a Farey pair. A similar argument demonstrates that $(d/n,1/k)$ is a Farey pair if $d$ divides $n-1$.

Conversely, if $d$ does not a divisor of either $n$ or $n-1$, then $dn - d(n-1) = d \neq 1$. The result now follows from \hyp{Lemma}{lem:farey_pair}.   
\end{proof}

\begin{cor}
\label{cor:divisor}
Let $d$, $m$, and $n$ be positive integers such that $d < m \leq n$, and suppose that $(1/m,1/m-1)$ is a Farey pair of order $n$. 
\begin{enumerate}[label=(\roman*)]
\item If $d$ divides $m$ and $k:=m/d$, then $(1/k, d/m-1)$ is a Farey pair of order $n$ if and only if $k + m - 1 > n$.   
\item If $d$ divides $m-1$ and $k:=(m-1)/d$, then $(d/m,1/k)$ is a Farey pair of order $n$ if and only if $m + k > n$.
\end{enumerate}
\end{cor}

%

\begin{thm} \label{thm:arcpowers}
Let $d$, $m$, and $n$ be positive integers such that $1< d < m \leq n$. Suppose that $(1/m,1/m-1)$ and $(d/m,d/m-1)$ are Farey pairs of order $n$. 
We distinguish the following cases:
\begin{enumerate}[label=(\roman*)]
\item d divides m: For $f_\alpha(t) = t^m - \beta t - \alpha$, let $M(\alpha)$ be defined as in \eqref{typeonemats}. If $\mathcal{M} := \{ M(\alpha) : \alpha \in [0,1] \}$, then $\mathcal{M}^d := \{ M(\alpha)^d : \alpha \in [0,1]\}$ forms a set of realizing-matrices for $K_n(1/k,d/m-1)$, where $k = m/d$.  

\item d divides $m-1$ and $m > k\floor{n/k}$, where $k = m/d$: For $f_\alpha(t) = t^m - \beta t - \alpha$, let $M(\alpha)$ be defined as in \eqref{typeonemats}. If $\mathcal{M} := \{ M(\alpha) : \alpha \in [0,1] \}$, then $\mathcal{M}^d := \{ M(\alpha)^d : \alpha \in [0,1]\}$ forms a set of realizing-matrices for $K_n(d/m,1/k)$, where $k = m/d$.
\end{enumerate}
\end{thm}

\begin{proof} Part (i): Since $(1/k,d/m-1)$ is a Farey pair, following \hyp{Corollary}{cor:divisor}, $n < m + k - 1$; consequently,  
\begin{align*}
d = \frac{m}{k} \leq \frac{n}{k} < \frac{m + k - 1}{k} = d + 1 - \frac{1}{k} < d + 1
\end{align*}
and hence $\floor{n/k} = d$. The Ito equations for $(1/k,d/m-1)$ are given by  
\begin{align*}
t^{m-1} \left( t^k - \beta \right)^{\floor{n/k}} = \alpha^{\floor{n/k}} t^{k\floor{n/k}},~\alpha \in [0,1],~\beta:=1-\alpha,
\end{align*}
and the reduced Ito polynomials for this arc are given by 
\begin{align}
q_\alpha (t) = (t^k - \beta)^d - \alpha^d t,~\alpha \in [0,1],~\beta:=1-\alpha. 
\end{align}
Notice that $\deg{(q_\alpha)} = m$, for every $\alpha \in [0,1]$. 

Let $\lambda = \lambda(\alpha) \in K(1/k,d/m-1)$. Consider the reduced Ito polynomial $p_\beta (t) = t^m -  \alpha t - \beta$ and its nonnegative companion matrix $M = M(\beta)$. The Cayley-Hamilton theorem (see, e.g., \cite[p.~109]{hj2013}) ensures that $M^m - \beta I = \alpha M$; hence  
\begin{align*}
q_\alpha(M^d) = (M^{dk} - \beta I)^d - \alpha^d M^d = (M^m - \beta I)^d  - (\alpha M)^d = 0,		
\end{align*}
i.e., $q_\alpha$ is an \emph{annihilating polynomial} for $M^d$.  

Denote by $\psi_M$ the \emph{minimal polynomial} of $M$, i.e., $\psi_M$ is the unique monic polynomial of minimum degree that annihilates $M$ (see, e.g., \cite[p.~192]{hj2013}). Since $M$ is a companion matrix, $\psi_M = \chi_M$ (\cite[Theorem 3.3.14]{hj2013}). Hence, if $J = \inv{S} M S$ is a Jordan canonical form of $M$, then $J$ is \emph{nonderogatory} (\cite[Theorem 3.3.15]{hj2013}), i.e., $J$ contains exactly one \emph{Jordan block} corresponding to every distinct eigenvalue. Since $M^d = S J^d \inv{S}$, it follows that any Jordan canonical form of $J^d$ is nonderogatory -- indeed, if $f(x) = x^d$, then $f'(x) = d x^{d-1}$ and $f'(x) = 0$ if and only if $x=0$; since zero is not a repeated root \eqref{lambdafrac} (and hence not associated with a nontrivial Jordan block), the claim follows from \cite[p.~424, Theorem 6.2.25]{hj1994}) -- thus, $M^d$ is nonderogatory and, following \cite[Theorem 3.3.15]{hj2013}, $\psi_{M^d} = \chi_{M^d}$ and $\deg{\left(\psi_{M^d}\right)} = m$. Since $\psi_{M^d}$ is the unique polynomial of minimum degree that annihilates $M$, and since $\deg{(q_\alpha)} = m$, it must be the case that $\chi_{M^d} = \psi_{M^d} = q_\alpha$. Hence, $M^d$ is a realizing-matrix for $\lambda$.

Part (ii): By hypothesis, 
\begin{align*}
d = \frac{m-1}{k} < \frac{m}{k} \leq \frac{n}{k}, 
\end{align*}
hence $d \leq \floor{n/k}$. Since $m > k\floor{n/k}$, it follows that $m - k\floor{n/k} \geq 1$ and $\floor{n/k} \leq (m-1)/k = d$. Hence, $d = \floor{n/k}$.  

The Ito equations for $(d/m,1/k)$ are given by 
\begin{align*}
t^m \left( t^k - \beta \right)^d = \alpha^d t^{m-1},~\alpha \in [0,1],~\beta:=1-\alpha, 
\end{align*}
and the reduced Ito polynomials for this arc are given by
\begin{align*}
q_\alpha (t) = t(t^k - \beta)^d - \alpha^d,~\alpha \in [0,1],~\beta:=1-\alpha.
\end{align*}
Notice that $\deg{(q_\alpha)} = m$, for every $\alpha \in [0,1]$. 

Let $\lambda = \lambda(\alpha) \in K(d/m,1/k)$. Consider the reduced Ito polynomial $f_\alpha (t) = t^m -  \beta t - \alpha$ and its nonnegative companion matrix $M = M(\alpha)$. The Cayley-Hamilton theorem ensures that $M(M^{m-1} - \beta I) = M^m - \beta M = \alpha I$; hence  
\begin{align*}
q_\alpha(M^d) = M^d (M^{m-1} - \beta I)^d - \alpha^d I = (M^m - \beta M)^d  - (\alpha I)^d = 0,		
\end{align*}
i.e., $q_\alpha$ is an \emph{annihilating polynomial} for $M^d$. 

Using exactly the same argument as in part (i), it can be shown that $\chi_{M^d} = \psi_{M^d} = q_\alpha$. Hence, $M^d$ is a realizing-matrix for $\lambda$. 
\end{proof}

\section{Additional Questions}

In this section, we pose several problems and conjectures for further inquiry.

\subsection{\karp~Arcs}

\hyp{Theorem}{thm:main} establishes the existence of parametric realizing-matrices for the K-arcs. Suppose that $M$ is a realizing-matrix for a given point on a given arc, and let $M_k$ be the irreducible component that realizes the arc. Clearly, $M_k^\top$ and $P M_k P^\top$ are also realizing-matrices. With the aforementioned in mind, we offer the following.  

\begin{prob}
To what extent are the realizing-matrices unique? 
\end{prob}

\hyp{Corollary}{cor:diffarcs} and \hyp{Theorem}{thm:arcpowers} show that many, but not all arcs are differentiable. Given the empirical evidence, we pose the following.

\begin{conj}
All K-arcs of order $n$ are differentiable for every $n$.
\end{conj} 

For $S\subseteq\bb{C}$, let $S^d := \{ \lambda^d : \lambda \in S \}$. \hyp{Theorem}{thm:arcpowers} demonstrates that $\sig{M}^d = \sig{M^d}$. Although the evidence is ample, a demonstration that the powered K-arc $K_n^d(1/m,1/m-1)$ corresponds to $K_n(1/k,d/m-1)$ ($d$ divides $m$) or $K_n(d/m,1/k)$ ($d$ divides $m-1$ and $m>k\floor{n/k}$) has proven elusive. Thus, we offer the following.

\begin{conj}
Let $d$, $m$, and $n$ be positive integers such that $1< d < m \leq n$. Suppose that $(1/m,1/m-1)$ and $(d/m,d/m-1)$ are Farey pairs of order $n$. 
\begin{enumerate}[label=(\roman*)]
\item If d divides m, then $K_n^d (1/m,1/m-1) = K_n(1/k,d/m-1)$, where $k = m/d$.  

\item If d divides $m-1$ and $m > k\floor{n/k}$, then $K_n^d (1/m,1/m-1) = K_n(d/m,1/k)$, where $k = m/d$. 
\end{enumerate}
\end{conj}

Let $K$ be a K-arc and let $d_K : [0,1] \longrightarrow \mathbb{R}_0^+$ be the function defined by $\alpha \longmapsto |\lambda|$, where $\lambda = \lambda(\alpha)$ is the point on $K$ corresponding to $\alpha \in [0,1]$. From \hyp{Figure}{fig:karpregions}, we pose the following. 

\begin{conj}
If $K$ is any K-arc, then the function $d_K$ is strictly convex. 
\end{conj}

\subsection{The Levick-Pereira-Kribs Conjecture}
 
For a natural number $n$, denote by $\Pi_n$ the convex-hull of the $n$\textsuperscript{th} roots-of-unity, i.e., 
\[ \Pi_n = \left\{ \sum_{k=0}^{n-1} \alpha_k \exp{(2\pi\ii k/n)} : \alpha_k \geq 0,~\sum_{k=0}^{n-1} \alpha_k =1 \right\}. \]
Denote by $\Omega_n$ the subset of the complex-plane containing all single eigenvalues of all $n$-by-$n$ doubly stochastic matrices. Perfect and Mirsky \cite{pm1965} conjectured that $\Omega_n = \bigcup_{k=1}^n \Pi_k$ and proved their conjecture when $1 \leq n \leq 3$. Levick et al.~\cite{lpk2015} proved Perfect-Mirsky when $n=4$ but a counterexample when $n=5$ was given by Mashreghi and Rivard \cite{mr2007}. Levick et al.~conjectured that $\Omega_n = \Theta_{n-1} \cup \Pi_n$ (\cite[Conjecture 1]{lpk2015}). 

In \cite{j1981}, necessary and sufficient conditions were found for a stochastic matrix to be similar to a doubly stochastic matrix. Thus, it is possible to investigate the Levick-Pereira-Kribs Conjecture via the realizing matrices given in \hyp{Theorem}{thm:main} vis-\`{a}-vis the results in \cite{j1981}. In particular, if $M$ is a realizing matrix for $\lambda$ on the boundary of $\Theta_n$ excluding the unit-circle (this case is clear), and $M \oplus 1$ is similar to a doubly stochastic matrix $D$, then $\Theta_{n-1} \cup \Pi_n \subseteq \Omega_n$.

\section{Acknowledgment}

We would like to thank University of Washington Bothell undergraduate student Amber R.~Thrall for proving that the polynomial $\pi$ is \hyp{Remark}{polypi} has only one root in $(0,1)$.


\bibliographystyle{abbrv}
\bibliography{master}

\end{document}